\documentclass{amsproc}

\usepackage{amssymb}
\usepackage{amsmath}
\usepackage{amsthm}
\usepackage{setspace}
\usepackage{enumitem}
\usepackage{mathtools}
\usepackage{mathrsfs}
\usepackage{pgfplots}
\usepackage[margin=1.5 in]{geometry}

\usepackage{pgf,tikz}
\usepackage{tikz-cd}
\usepackage{mathrsfs}
\usetikzlibrary{arrows}

\newtheorem{theorem}{Theorem}[section]
\newtheorem{lemma}[theorem]{Lemma}
\newtheorem{corollary}[theorem]{Corollary}
\newtheorem{prop}[theorem]{Proposition}

\theoremstyle{definition}
\newtheorem{definition}[theorem]{Definition}

\theoremstyle{remark}

\numberwithin{equation}{section}

\begin{document}

\title{Quasi-isometric rigidity of a class of right-angled Coxeter groups}

\author{Jordan Bounds}
\address{Department of Mathematics, Bowling Green State University, Bowling Green, Ohio;}
\email{boundsj@bgsu.edu}

\author{Xiangdong Xie}
\address{Department of Mathematics, Bowling Green State University, Bowling Green, Ohio;}
\email{xiex@bgsu.edu}

\subjclass[2010]{Primary }

\date{}

\dedicatory{}

\begin{abstract}
	We establish quasi-isometric rigidity for a class of 
	right-angled Coxeter groups. Let $\Gamma_1$, $\Gamma_2$ be joins of finite generalized thick $m$-gons with $m\ge 3$. We show that the corresponding right-angled Coxeter groups are quasi-isometric if and only if $\Gamma_1$, $\Gamma_2$ are isomorphic. We also give a construction of commensurable right-angled Coxeter groups.
\end{abstract}

\maketitle

\begin{section}{Introduction}
	
	In this note we study the large scale geometry of right-angled Coxeter groups  (RACGs in short). 
	Specifically we establish the quasi-isometric rigidity of a class of  RACGs.

	Given a finite simplicial graph $\Gamma$ with vertex set $V(\Gamma)$ and  edge set $E(\Gamma)$, there is  an associated RACG $W_\Gamma$ 
	given by the presentation
	\[W_\Gamma=\langle v\in V(\Gamma)|     v^2=1 \;\text{for all}\; v\in V(\Gamma);   v_1v_2=v_2v_1\; \text{if and only if }\;  (v_1, v_2)\in E(\Gamma)\rangle.\]
	The Davis complex $\Sigma_\Gamma$ associated to  $W_\Gamma$ is a  CAT($0$) cube complex upon which $W_\Gamma$ acts geometrically (that is, properly and cocompactly by isometries). The Davis complex should be viewed as a  geometric model for the RACG. See Section 2.1 for its definition.
	
	Let $L$ be a connected bipartite graph whose vertices are colored red and blue such that no two adjacent vertices share the same color. $L$ is called a {\em generalized m-gon}     
	$(m\in\mathbb{N},m\geq 2)$ if it has the following two properties:\begin{enumerate}
		\item Given any pair of edges there is a circuit with combinatorial length $2m$ containing both.
		\item For two circuits $A_1,A_2$ of combinatorial length $2m$ that share an edge there is an isomorphism $f:A_1\rightarrow A_2$ that pointwise fixes $A_1\cap A_2$.
		
	\end{enumerate}
	A generalized $m$-gon is called {\em thick} if each vertex has valence at least 3. There are many finite generalized thick polygons, see Section 2.2 for more information. Our main result is as follows.
	
	\begin{theorem}\label{main}
		For $i=1,2$, let $\Gamma_i$ be a finite generalized thick $m_i$-gon with $m_i\ge 3$.
		Then any quasi-isometry $f: \Sigma_{\Gamma_1}\rightarrow \Sigma_{\Gamma_2}$ is at a finite distance from an  isometry. In particular,   $W_{\Gamma_1}$ and 
		$W_{\Gamma_2}$  are quasi-isometric if and only if $\Gamma_1$ and $\Gamma_2$ are isomorphic. 
	\end{theorem}
	
	As an immediate consequence, we have the following corollary:
	
	\begin{corollary}\label{cor1}
		Let $\Gamma$ be a finite generalized thick $m$-gon with $m\ge 3$.   If a finitely generated group $G$ is quasi-isometric to $W_\Gamma$, then $G$ acts geometrically on $\Sigma_\Gamma$.
	\end{corollary}

	For each $m\ge 3$, let ${\mathcal S}_m$ be the collection of finite generalized thick $m$-gons.
	Let $\mathcal S=\cup_m \mathcal S_m$.
	By combining Theorem \ref{main} with a result of  Kapovich-Kleiner-Leeb \cite{KKL98}
	on the   quasi-isometric rigidity of product spaces, we obtain:
	
	\begin{corollary}\label{cor2}
		Let $\Gamma_1$, $\Gamma_2$ be finite joins of graphs from $\mathcal S$. Then any quasi-isometry
		$f: \Sigma_{\Gamma_1}\rightarrow  \Sigma_{\Gamma_2}$ is at a finite distance from an
		isometry. In particular,   $W_{\Gamma_1}$ and 
		$W_{\Gamma_2}$  are quasi-isometric if and only if $\Gamma_1$ and $\Gamma_2$ are isomorphic. 
	\end{corollary}
	
	The key point in the proof of  Theorem \ref{main}   is  that  when $\Gamma$ is a finite generalized thick $m$-gon with $m\ge 3$ 
	$\Sigma_\Gamma$ admits a metric structure that makes it a Fuchsian building. Theorem \ref{main} then follows from the  quasi-isometric rigidity for Fuchsian buildings \cite{Xie06}.    Under the assumptions 
	of Theorem \ref{main}  the Davis complex is a CAT($0$) square complex, where each $2$-cell is isometric to a Euclidean unit square.   We replace each square with a regular $4$-gon in the hyperbolic plane with angle $\pi/m$ at each vertex.   The Davis complex with this piecewise hyperbolic metric is then a Fuchsian building (see Theorem \ref{Thm}).
	
	There exist non-isomorphic graphs $\Gamma_1$, $\Gamma_2$ such that $W_{\Gamma_1}$, $W_{\Gamma_2}$ are quasi-isometric. In Section 5, we give a construction that produces commensurable RACGs. Another way to construct quasi-isometric RACGs is to use the ``fattened tree'' method of Behrstock-Neuman (see \cite{DT14}, \cite{BN08}, \cite{Mal10}, \cite{CM17}).
	
	The quasi-isometric classification of RACGs is wide open at this point. Here we give a very brief  summary of the current state of the art. For more details see the survey article by Pallavi Dani \cite{D17}.  By a result of Behrstock-Caprace-Hagen-Sisto \cite{BHS17}, the RACGs are divided into two classes, those that are
	algebraically thick, and those that are relatively hyperbolic.    Algebraically thick ones are not quasi-isometric to relatively hyperbolic ones.   Divergence is a quasi-isometric invariant for finitely generated groups.  Algebraically thick  RACGs have divergence bounded above by a polynomial function \cite{BD14}, and relatively hyperbolic RACGs have exponential divergence function  (\cite{S12}, Theorem 1.3).  It is still an open question whether the divergence of an
	algebraically thick  RACG   is exactly polynomial. 
	In \cite{NT17},  Nguyen-Tran  classified up to quasi-isometry a class of algebraically thick RACGs whose graphs  are planar  $\mathcal{CFS}$ graphs.   Dani-Thomas \cite{DT14}   classified up to quasi-isometry a class of 2 dimensional hyperbolic 
	RACGs that split over $2$-ended subgroups.     Results from \cite{DT14} and 
	\cite{HST17} together  classified  up to quasi-isometry all RACGs whose graphs are generalized theta graphs.  
	
	The paper is structured as follows. We recall some preliminaries about RACGs, generalized polygons and Fuchsian buildings in Section 2. In Section 3 we show that the Davis complex associated to the graph of a generalized $m$-gon admits the structure of a labeled 2-complex. In Section 4 we prove Theorem \ref{main} and Corollary \ref{cor2}. We conclude with Section 5 where we describe a method for constructing examples of commensurable, and therefore quasi-isometric, RACGs. \\ 
	
	\noindent {\bf{Acknowledgment}}. {The second author   acknowledges support from Simons Foundation grant \#315130.  }
	
\end{section}
\begin{section}{Preliminaries}
	
	\begin{subsection}{Right-angled Coxeter groups}
		
		A graph is called {\em simplicial} if it has no loops (an edge whose initial and terminal vertices are equal) and there is at most one edge between every pair of vertices. An {\em n-clique} is the complete graph on $n$ vertices. 
		
		\begin{definition} Given a finite simplicial graph $\Gamma$ with vertex set $V(\Gamma)$ and edge set $E(\Gamma)$ the associated {\em right-angled Coxeter group} (RACG) $W_\Gamma$ is the group defined by the following presentation:
			\[W_\Gamma=\langle v\in V(\Gamma)|     v^2=1 \;\text{for all}\; v\in V(\Gamma);   v_1v_2=v_2v_1\; \text{if and only if }\;  (v_1, v_2)\in E(\Gamma)\rangle.\]
			By a theorem of Moussong $W_\Gamma$ is Gromov hyperbolic if and only if $\Gamma$ has no induced 4-cycles \cite{Mou88}. It is also well known that $\Gamma$ decomposes as a graph join $\Gamma=\Gamma_1\star\Gamma_2$ if and only if $W_\Gamma=W_{\Gamma_1}\times W_{\Gamma_2}$.
			
		\end{definition}
		To each RACG $W_\Gamma$ there is an associated CAT($0$) cube complex upon which $W_\Gamma$ acts geometrically called the {\em Davis complex associated to} $W_\Gamma$, denoted by $\Sigma_\Gamma$. Its original construction and definition can be found in \cite{Dav08}. We provide a summary of the construction here as seen in \cite{D17}.
		
		\begin{definition}
			Let $\Gamma$ be a finite simplicial graph with $n$ vertices. Let $C=[-1,1]^n\subset\mathbb{R}^n$. Fix an ordering on the vertex set $V(\Gamma)$ of $\Gamma$ so that the $i$th generator of $W_\Gamma$ corresponds to the $i$th factor of $C$. Then each face of $C$ determines a unique subset of $V(\Gamma)$. Let $P_L$ be the subcomplex of $C$ which contains a face $F\subset C$ if and only if the subset of $V(\Gamma)$ that corresponds to $F$ forms a clique in $\Gamma$. The Davis complex $\Sigma_\Gamma$ associated to $W_\Gamma$ is the universal cover of $P_L$.
			
		\end{definition}
		We note that $\Sigma_\Gamma$ is a two dimensional cube complex if and only if the defining graph $\Gamma$ is triangle free and contains at least one edge.

	\end{subsection}
	\begin{subsection}{Generalized polygons}
		We now briefly recall some important facts regarding generalized polygons as defined in the introduction. For a generalized thick $m$-gon, we often call an edge a {\em chamber} and a circuit with combinatorial length $2m$ an {\em apartment}. Two vertices of the same color are said to have the same type.
		
		For $m\geq 3$ there exist many examples of finite generalized $m$-gons (see \cite{Mal} Section 1.7). A well known theorem of Feit-Higman \cite{FH64} says that if $L$ is a finite generalized thick $m$-gon, then $m$ is one of \{2,3,4,6,8\}. For a generalized thick $m$-gon, there are integers $r,b\geq 2$ such that each red vertex is contained in exactly $r+1$ chambers and each blue vertex is contained in exactly $b+1$ chambers (see p.29 of \cite{MR89}). It is known that $r=b$ in the case where $m$ is odd and $r\neq b$ when $m=8$.
		
		Let $L$ be a generalized $m$-gon. For each chamber in $L$ we equip a metric such that it is isometric to the closed interval of length $\pi/m$. We then equip $L$ with the path metric, making $L$ a CAT($1$) space. Every apartment $A$ in $L$ is then convex in $L$, meaning that for any $x,y\in A$ with $d(x,y)<\pi$ the geodesic segment $xy$ also lies in $A$. Thus, if two apartments $A_1,A_2$ share a chamber, then either $A_1=A_2$ or $A_1\cap A_2$ is a segment.
		
	\end{subsection}
	\begin{subsection}{Fuchsian buildings}
		
		Let $R$ denote a compact convex polygon in $\mathbb{H}^2$ whose angles are of the form $\pi/m$, $m\in\mathbb{N}$, $m\geq 2$. Let $W$ be the Coxeter group generated by the reflections about the edges of $R$. 
		We label the edges and vertices of $R$ cyclically by $\{1\}, \{2\},\dots, \{k\}$ and $\{1,2\}, \dots, \{k-1,k\}, \{k,1\}$ respectively so that the edges $i$ and $i+1$ intersect at the vertex $\{i,i+1\}$. It is well-known that the images of $R$ under $W$ form a tessellation of $\mathbb{H}^2$ and that the quotient $\mathbb{H}^2/W=R$. Thus, there is a labeling of edges and vertices of the tessellation of $\mathbb{H}^2$ that is $W$-invariant and compatible with the labeling of $R$. Let $A_R$ denote the obtained labeled 2-complex.
		
		\begin{definition}\label{fuch}
			Let $\Delta$ be a connected cellular 2-complex whose edges and vertices are labeled by $\{1\}, \{2\},\dots,\{k\}$ and $\{1,2\},\{2,3\},\dots\{k-1,k\},\{k,1\}$ respectively, such that each 2-cell (called a {\em chamber}) is isomorphic to $R$ as labeled 2-complexes. We call $\Delta$ a {\em two dimensional hyperbolic building} if it has a family of subcomplexes (called apartments) isomorphic to $A_R$ (as labeled 2-complexes) satisfying the following properties:
			\begin{enumerate}
				\item Given any two chambers there is an apartment containing both.
				\item For any two apartments $A_1,A_2$ that share a chamber there is an isomorphism of labeled 2-complexes $f:A_1\rightarrow A_2$ which pointwise fixes $A_1\cap A_2$.
			\end{enumerate}
			If in addition there are integers $q_i\geq 2, i=1,2,\dots, k$, such that each edge of $\Delta$ labeled by $i$ is contained in exactly $q_i+1$ chambers, then $\Delta$ is called a {\em Fuchsian building}.
			
		\end{definition}
		The following result established in \cite{Xie06} is due to one of the authors.
		\begin{theorem}\label{xie}
			Let $\Delta_1,\Delta_2$ be two Fuchsian buildings and $g:\Delta_1\rightarrow\Delta_2$ a quasi-isometry. If $\Delta_1,\Delta_2$ admit cocompact lattices, then $g$ lies at a finite distance from an isomorphism.
		\end{theorem}

	\end{subsection}
\end{section}

\begin{section}{The Davis complex as a labeled 2-complex}
	In this section we show that the Davis complex $\Sigma_\Gamma$ associated to the graph of a generalized $m$-gon $\Gamma$ has the structure of a labeled $2$-complex. Throughout, given $i\in\{1,2,3,4\}$ we will set $i+1=1$ if $i=4$.
	
	\begin{prop}\label{prop}
		Let $\Gamma$ be the graph of a generalized $m$-gon. Then there is a labeling of the edges of $\Sigma_\Gamma$ by $\{1\},\{2\},\{3\},\{4\}$ and a labeling of the vertices by $\{1,2\}, \{2,3\}, \{3,4\}, \{4,1\},$ such that for every 2-cell $S$ in $\Sigma_\Gamma$ its edges are cyclically labeled by $1,2,3,4$ and the vertex of $S$ incident to the edges labeled by $i,i+1$ is labeled by $\{i,i+1\}$.
	\end{prop}
	Proving this statement is the first step in showing that $\Sigma_\Gamma$ admits a Fuchsian building structure. As $\Sigma_\Gamma$ is the universal cover of $P_L$, it suffices to show that $P_L$ admits the structure of a labeled 2-complex. One can then lift the labeled 2-complex structure on $P_L$ to obtain a labeled 2-complex structure on $\Sigma_\Gamma$. Recall the construction of $P_L$.  A one-to-one  correspondence was fixed between the set of coordinate axes of $\mathbb R^n$ and the vertex set $V(\Gamma)$ of $\Gamma$. The vertices of $\Gamma$ are divided into red  vertices and blue vertices.   A coordinate axis is red (blue) if it corresponds to a red (blue) vertex. An edge $e$ of $P_L$ is red (blue) if it is parallel to a red (blue) coordinate axis.
	
	We begin by defining a labeling of $P_L$ and establishing a few preliminary results. For any vertex $v$ of $P_L$ let $E_v$ be the set of all edges in $P_L$ incident to $v$.

	\begin{definition}\label{PL}
		A {\em labeling} of $P_L$ is a map $l$ from the edge set of $P_L$ to the set $\{1,2,3,4\}$ with the following two properties:
		\begin{enumerate}
			\item For any 2-cell $S$ in $P_L$, the edges of $S$ are cyclically labeled 1,2,3,4.
			\item For any vertex $v$, $l(E_v)=\{i,i+1\}$ for some $i\in\{1,2,3,4\}$.
		\end{enumerate}
	\end{definition}
	
	Given any labeling $l$ of $P_L$ we obtain a labeled 2-complex structure of $P_L$ by labeling each vertex $v$ of $P_L$ with $l(E_v)$. To obtain a labeling of $P_L$ we begin with a labeling of $E_{v_0}$ for a fixed vertex $v_0$ and then use compatibility (see Definition \ref{compatible}) to obtain a labeling of $E_v$ for every vertex $v$.

	\begin{definition}\label{linklabel}
		For a vertex $v\in P_L$ a {\em labeling of $E_v$} is a map $l$ from $E_v$ to the set $\{1,2,3,4\}$ with the following properties:\begin{enumerate}
			\item $l(E_v)=\{i,i+1\}$ for some $i\in\{1,2,3,4\}$.
			\item For all $e_1,e_2\in E_v$, $l(e_1)=l(e_2)$ if and only if $e_1,e_2$ have the same color.
		\end{enumerate}
	\end{definition}
	Labelings on $E_v$ clearly exist as one can define $l(e)=i$ for all red edges in $E_v$ and $l(e)=i+1$ for all blue edges in $E_v$.

	\begin{definition}\label{compatible}
		
		Let $v_1,v_2$ be two adjacent vertices in $P_L$ and $l_1,l_2$ labelings of $E_{v_1},E_{v_2}$ respectively. Let $e$ be the edge with vertices $v_1,v_2$. We say  $l_1$ and $l_2$ are {\em compatible} if there is some $i\in\{1,2,3,4\}$ such that  $l_1(e)=l_2(e)=i$ and either $l_1(E_{v_1})=\{i-1, i\}$, $l_2(E_{v_2})=\{i, i+1\}$ or  $l_2(E_{v_2})=\{i-1, i\}$, $l_1(E_{v_1})=\{i, i+1\}$.
		
	\end{definition}
	Let $v_1,v_2$ be adjacent vertices in $P_L$. Note that given any labeling $l_1$ of $E_{v_1}$ there is exactly one labeling $l_2$ of $E_{v_2}$ that is compatible with $l_1$. Indeed,  let $l_1$ be a labeling of $E_{v_1}$ and $e_0$ the edge  with $v_1$, $v_2$ as vertices. Without loss of generality we may assume $l_1(e)=i$ for red edges and $l_1(e)=i+1$ for blue edges  in $E_{v_1}$.  If $e_0$ is red, then the unique labeling  $l_2$ of $E_{v_2}$ compatible with $l_1$ is given by: $l_2(e)=i$  if $e$ is red and $l_2(e)=i-1$  if $e$ is blue. Similarly if  $e_0$ is blue, then the  unique labeling  $l_2$ of $E_{v_2}$ compatible with $l_1$ is given by: $l_2(e)=i+1$  if $e$ is blue and $l_2(e)=i+2$ if $e$ is red. 
	
	We now outline how to construct a labeling of $P_L$. Fix a vertex $v$ of $P_L$ and a labeling $l$ of $E_v$. As previously stated, given adjacent vertices $v_1,v_2$ and a labeling $l_1$ of $E_{v_1}$ there is a unique labeling $l_2$ of $E_{v_2}$ compatible with $l_1$. Now if $C=e_1e_2\dots e_m$ is an edge path from $v$ to some vertex $v'$, then using compatibility along $C$ we obtain a labeling $l'$ of $E_{v'}$. As there may be many edge paths from $v$ to $v'$, we must show that $l'$ is independent of the chosen path $C$. Equivalently we need to show that if $C$ is an edge loop at $v$, then the labeling $l'$ of $E_v$ obtained from $l$ by compatibility along the loop $C$ coincides with $l$. We show this by first considering the case when $C$ is a $4$-cycle (Lemma \ref{4-cycle}). We then write a general edge loop as the concatenation of conjugates of $4$-cycles (Lemma \ref{circuit}).

	\begin{figure}
		\begin{tikzpicture}[scale=0.45];
		\draw [line width=1pt] (-2,2)-- (-2,-2);
		\draw [line width=1pt] (-2,-2)-- (2,-2);
		\draw [line width=1pt] (2,-2)-- (2,2);
		\draw [line width=1pt] (2,2)-- (-2,2);
		\begin{scriptsize}
		\draw [fill=black] (-2,2) circle (2.5pt);
		\draw[color=black] (-1.78,2.3) node {$v_4$};
		\draw [fill=black] (-2,-2) circle (2.5pt);
		\draw[color=black] (-1.9,-2.3) node {$v$};
		\draw[color=black] (-2.3,0.25) node {$\bar e_2$};
		\draw [fill=black] (2,-2) circle (2.5pt);
		\draw[color=black] (2.1,-2.3) node {$v_2$};
		\draw[color=black] (0.08,-2.3) node {$e_1$};
		\draw [fill=black] (2,2) circle (2.5pt);
		\draw[color=black] (2.22,2.3) node {$v_3$};
		\draw[color=black] (2.4,0.25) node {$e_2$};
		\draw[color=black] (0.08,2.57) node {$\bar e_1$};
		\end{scriptsize}
		\end{tikzpicture}	
		\caption{}\label{cyclepic}
	\end{figure}

	For any oriented edge $e$ of $P_L$ let $e^{-1}$ denote the same edge given with the opposite orientation, $\tilde e$ any edge parallel to $e$ with the same direction, and $\bar e$ any edge parallel to $e$ with the opposite orientation. In particular if $\tilde e,\bar e$ occur consecutively in an edge path of $P_L$, then this $\bar e$ is equal to $\tilde e^{-1}$.

	\begin{lemma}\label{4-cycle}
		Let $C=e_1e_2\bar e_1\bar e_2$ be a $4$-cycle in $P_L$ based at $v$. Then for any labeling $l$ of $E_v$ we have $l'=l$ where $l'$ is the labeling of $E_v$ obtained from $l$ by compatibility along $C$.
	\end{lemma}
	
	\begin{proof}
		Suppose $v$ is the vertex in $C$ incident to both $e_1$ and $\bar e_2$ and let $l$ be a labeling of $E_v$. Denote the remaining vertices in $C$ as in Figure \ref{cyclepic}. We have two cases to consider. 
		
		Case 1: Suppose $e_1,\bar e_2$ are of the same color. As $\bar e_1$ and $e_2$ are parallel to $e_1$ and $\bar e_2$ respectively, all of $e_1,e_2,\bar e_1,\bar e_2$ are of the same color. Thus there is some $i\in\{1,2,3,4\}$ such that $l(e_1)=l(\bar e_2)=i$. Furthermore either $l(E_v)=\{i,i+1\}$ or $l(E_v)=\{i,i-1\}$. Without loss of generality we may assume $l(E_v)=\{i,i+1\}$. Let $l_2$ be the labeling of $E_{v_2}$ obtained from $l$ by compatibility along $e_1$. Then $l_2(E_{v_2})=\{i-1,i\}$ and $l_2(e_1)=l_2(e_2)=i$. Similarly, let $l_3$ be the labeling of $E_{v_3}$ obtained from $l_2$ by compatibility along $e_2$. Then $l_3(E_{v_3})=\{i,i+1\}$ and $l_3(e_2)=l_3(\bar e_1)=i$. Next let $l_4$ be the labeling of $E_{v_4}$ obtained from $l_3$ by compatibility along $\bar e_1$. Then $l_4(E_{v_4})=\{i-1,i\}$ and $l_4(\bar e_1)=l_4(\bar e_2)=i$. Finally, we obtain the labeling $l'$ of $E_v$  from $l_4$ by compatibility along $\bar e_2$. It follows that $l'(E_v)=\{i,i+1\}$ and $l'(\bar e_2)=l'(e_1)=i$. Thus $l'=l$.
		
		Case 2: Suppose $e_1,\bar e_2$ are of different colors. Then there must be some $i\in\{1,2,3,4\}$ such that either $l(e_1)=i$, $l(\bar e_2)=i+1$ or $l(e_1)=i+1,l(\bar e_2)=i$. Without loss of generality, suppose $l(e_1)=i$, $l(\bar e_2)=i+1$. Hence $l(E_v)=\{i,i+1\}$. Let $l_2$ be the labeling of $E_{v_2}$ obtained from $l$ by compatibility along $e_1$. Then $l_2(e_1)=i$ and $l_2(E_{v_2})=\{i-1,i\}$. As $e_1,\bar e_2$ are of different colors, so are $e_1,e_2$. Then $l_2(e_1)\neq l_2(e_2)$ must be the case. It then follows that $l_2(e_2)=i-1$. Now let $l_3$ be the labeling of $E_{v_3}$ obtained from $l_2$ by compatibility along $e_2$. Then $l_3(e_2)=i-1$ and so $l_3(E_{v_3})=\{i-2,i-1\}$ by property 2 of Definition \ref{linklabel}. As before, $e_2,\bar e_1$ will be of different colors, implying that $l_3(\bar e_1)=i-2$. Similarly, let $l_4$ be the labeling of $E_{v_4}$ obtained from $l_3$ by compatibility along $\bar e_1$. Then $l_4(\bar e_1)=i-2$, $l_4(E_{v_4})=\{i-3,i-2\}$, and $l_4(\bar e_2)=i-3$. Finally, let $l'$ be the labeling of $E_v$ obtained from $l_4$ by compatibility along $\bar e_2$. It follows that $l'(\bar e_2)=l_4(\bar e_2)=i-3$, $l'(E_v)=\{i-4,i-3\}$, and $l'(e_1)=i-4$. By our previously stated convention, $i-4=i$ and $i-3=i+1$, hence $l'=l$.
	\end{proof}
	We now define an equivalence relation $\sim$ on the edge paths of $P_L$. We say that two edge paths $C,C'$ satisfy $C\sim C'$ if $C'$ can be obtained from $C$ by inserting and deleting pairs of the form $ee^{-1}$. In this case $C$ and $C'$ have the same initial and terminal vertices. Observe that if $v_0$ is the initial vertex, $l_0$ is a labeling of $E_{v_0}$ and $l$,$l'$ are the labelings at the terminal vertex obtained from $l_0$ by compatibility along $C$ and $C'$ respectively, then $l=l'$.
	\begin{lemma}\label{circuit}
		Let $C=e_1\dots e_{2n}$ be an edge loop in $P_L$. Then there exists an edge loop $C'=f_1\dots f_k$ in $P_L$ with $C'\sim C$ such that each $f_j$ is of the form $\alpha_j(\tilde e_s\tilde e_t\bar e_s\bar e_t)\alpha_j^{-1}$ with $1\leq s<t\leq 2n$ and $\alpha_j$ an edge path of $P_L$.
	\end{lemma}
	\begin{proof}
		By deleting pairs of the form $ee^{-1}$, we may assume 
		$e_{i+1}\not=e_i^{-1}$ for any $i$. We induct on  the length $2n$ of the edge loop $C$. Taking $C$ to have length $4$ as our base case, $C=e_1e_2\bar e_1\bar e_2$ already has the required form. Suppose the statement holds for  edge loops with length $<2n$. Now assume 
		the edge loop $C$ has length $2n$.
		
		As $C$ is an edge loop in $P_L$ there exists $j$ such that $e_j=\bar e_1$. We then have that \begin{align*}
		C&=e_1e_2e_3\cdots e_j\cdots e_{2n}\\
		&\sim (e_1e_2\bar e_1\bar e_2)\bar e_2^{-1}\tilde e_1e_3\cdots e_j\cdots e_{2n},\hspace{1.9cm}\text{with}\ \tilde e_1=\bar e_1^{-1},\\
		&=f_1\bar e_2^{-1}\tilde e_1e_3\cdots e_j\cdots e_{2n}, \hspace{3cm}\text{with}\ f_1=e_1e_2\bar e_1\bar e_2,\\
		&\sim f_1\bar e_2^{-1} (\tilde e_1e_3\bar e_1\bar e_3)\bar e_2\bar e_2^{-1}\bar e_3^{-1}\tilde e_1e_4\cdots e_j\cdots e_{2n}\\
		&=f_1f_2\bar e_2^{-1}\bar e_3^{-1}\tilde e_1e_4\cdots e_j\cdots e_{2n}, \hspace{2.1cm}\text{with}\ f_2=\bar e_2^{-1} (\tilde e_1e_3\bar e_1\bar e_3)\bar e_2,\\
		&\cdots\\
		&\sim f_1f_2\cdots f_{j-2}\bar e_2^{-1}\cdots\bar e_{j-1}^{-1}\tilde e_1e_j\cdots e_{2n}\\
		&\sim f_1f_2\cdots f_{j-2}\bar e_2^{-1}\cdots\bar e_{j-1}^{-1}e_{j+1}\cdots e_{2n}, \hspace{1.4cm}\text{since}\ e_j=\bar e_1.
		\end{align*} 
		The length of the edge loop $\bar e_2^{-1}\cdots\bar e_{j-1}^{-1}e_{j+1}\cdots e_{2n}$ is less than $2n$. Thus it follows inductively that there exist $f_{j-1},\dots,f_k$, each of the desired form, such that $C\sim f_1\cdots f_jf_{j+1}\cdots f_k$.
	\end{proof}
	
	We now prove Proposition \ref{prop}.
	
	\begin{proof}[Proof of Proposition \ref{prop}]
		As already observed, it suffices to construct a labeling of $P_L$ as this induces a labeled 2-complex structure on $P_L$ which then lifts to a labeled 2-complex structure on $\Sigma_\Gamma$. Fix a vertex $v_0$ of $P_L$ and a labeling $l_0$ of $E_{v_0}$. By Lemmas \ref{4-cycle} and \ref{circuit}, we obtain a labeling $l_v$ of $E_v$ for every vertex $v$ of $P_L$ such that $l_{v_1},l_{v_2}$ are compatible whenever $v_1,v_2$ are adjacent. We define a labeling $l$ of $P_L$ as follows. For every edge $e$ of $P_L$ define $l(e)=l_{v_1}(e)$ where $v_1$ is a vertex of $e$. Note that if $v_2$ is the other vertex of $e$, then by compatibility we have $l_{v_2}(e)=l_{v_1}(e)$. Thus $l$ is well-defined. Property 2 in Definition \ref{PL} is clearly satisfied. We show that property 1 holds as well.
		
		Let $S$ be a 2-cell in $P_L$. Denote its vertices by $v_1,v_2,v_3,v_4$ and edges by $e_1,e_2,e_3,e_4$ such that the edge $e_i$ is incident to both $v_i$ and $v_{i+1}$. Let $l_{v_i}$ be obtained labelings of $E_{v_i},$ $ i=1,2,3,4$. First note that each pair of adjacent edges $e_i,e_{i+1}$ must have different colors. We proceed with an argument similar to that of Case 2 in the proof of Lemma \ref{4-cycle}. Given that the labelings obtained via compatibility are independent of chosen edge path, $l_{v_{i+1}}$ can be obtained from $l_{v_i}$ by compatibility along $e_i$, $i=1,2,3,4$. Without loss of generality we may assume $l_1(E_{v_1})=\{i,i+1\}$ and $l_1(e_1)=i$. By compatibility along $e_1$ it follows that $l_2(E_{v_2})=\{i-1,i\}$, $l_2(e_1)=i$, and $l_2(e_2)=i-1$. By compatibility along $e_2$ we obtain $l_3(E_{v_3})=\{i-2,i-1\}$, $l_3(e_2)=i-1$, and $l_3(e_3)=i-2$. Finally, $l_4(E_{v_4})=\{i-3,i-2\}$, $l_4(e_3)=i-2$, and $l_4(e_4)=i-3$. Thus, the boundary of $S$ is labeled cyclically as $1,2,3,4$.
	\end{proof}

	For any edge $e$ of $P_L$, the degree $d(e)$  of $e$ is defined to be the  number of squares in $P_L$ that contain $e$. We next show that edges with the same label have the same degree (Lemma 3.7 (3)). This result is needed to show that $\Sigma_\Gamma$ admits a Fuchsian building structure.

	\begin{lemma}\label{degree} Let $l$ be a labeling of $P_L$ and $e, e'$ be two edges of $P_L$.\begin{enumerate}
			\item If $l(e), l(e')$ have the same parity, then  $e, e'$ have the same color;
			\item If  $l(e), l(e')$ have different  parity, then  $e, e'$ have different  colors;
			\item   If   $l(e), l(e')$ have the same parity, then  $d(e)=d(e')$. In particular,
			edges with the same label have the same degree.
		\end{enumerate}
	\end{lemma}
	
	\begin{proof} (1) and (2):  We prove  (1)  and (2) at the same time by inducting
		on  the minimal length $m$ of edge paths from $e$ to  $e'$. In the case where $m=0$, $e$ and $e'$ share a vertex.  If $l(e), l(e')$ have the same parity,   then $l(e)=l(e')$  and $e, e'$ are of the same color. If
		$l(e), l(e')$ have different  parity,  then there is some $i$ such that $l(e)=i$, $l(e')=i+1$ or  $l(e)=i+1$, $l(e')=i$  and so  $e, e'$ are of different colors. Assume (1) and (2) hold for $m$. Now suppose the minimal   length of an edge path from $e$ to $e'$ is $m+1$. Let $e e_1\cdots e_{m+1} e'$ be such an edge path.  First suppose $l(e), l(e')$ have the same parity.  If $l(e_1)$ has the same parity as $l(e)$, then $l(e_1),l(e')$ have the same parity. By the inductive assumption $e_1,e'$ are of the same color, $e,e_1$ are of the same color, and so $e,e'$ are of the same color. If $l(e_1)$, $l(e)$  have different parity, then $l(e_1),l(e')$ have different parity. By the inductive assumption $e_1,e'$ are of different colors, $e,e_1$ are of different colors, and so $e,e'$ are of the same color.  Next suppose  $l(e), l(e')$ have different  parity. If $l(e),l(e_1)$ have the same parity, then $l(e_1),l(e')$ have different parity. Then by the inductive assumption $e_1,e'$ are of different colors, $e,e_1$ are of the same color, and so $e,e'$ are of different colors. If $l(e),l(e_1)$ have different parity, then $l(e_1),l(e')$ have the same parity. Again by the inductive assumption $e_1,e'$ are of the same color, $e,e_1$ are of different color, and so $e,e'$ are of different colors.
		
		(3):  Suppose  $l(e), l(e')$ have the same parity.   By (1)
		$e$, $e'$ have the same color.   Let $u, u'$ be vertices of $\Gamma$ that correspond to $e$ and $e'$ respectively.
		Then $u, u'$ have the same color. Also observe
		$d(e)=d(u)$ (valence of $u$) and $d(e')=d(u')$.   Since vertices of the same color in $\Gamma$ have the same valence, we have $d(u)=d(u')$. It follows that $d(e)=d(e')$.
	\end{proof}
\end{section}

\begin{section}{Proof of main results}
	In this section we prove Theorem 1.1 and Corollary 1.3. Throughout, let $\mathcal S_m,m\geq 3$, be the collection of finite generalized thick $m$-gons. Let $\mathcal S=\cup_m \mathcal S_m$ and $\Gamma\in\mathcal S$.  As indicated in the introduction,  to
	prove Theorem 1.1, we first show that $\Sigma_\Gamma$ admits a 
	Fuchsian building structure and then apply the quasi-isometric rigidity theorem for Fuchsian buildings (Theorem 2.4).  In section 3 we already showed that $\Sigma_\Gamma$ admits the structure of a labeled 2-complex and that all edges with the same label have the same degree.  To show that $\Sigma_\Gamma$ admits a 
	Fuchsian building structure, it remains to put a piecewise hyperbolic metric on $\Sigma_\Gamma$ so that it satisfies conditions (1) and (2) in the definition of a Fuchsian building. 
	
	We equip the Davis complex $\Sigma_\Gamma$ with a piecewise hyperbolic metric by replacing each Euclidean square with a regular four sided polygon in $\mathbb{H}^2$ that has angle $\pi/m$ at each vertex. Observe that every vertex link is $\Gamma$ where each edge has length $\pi/m$ and so is CAT($1$). By Gromov's link condition, $\Sigma_\Gamma$ with this piecewise hyperbolic metric is a CAT($-1$) space.

	To verify condition (1) of a Fuchsian building, we let $P$ and $Q$ be two squares in $\Sigma_\Gamma$ and let $p,q$ be interior points of $P$ and $Q$ respectively. Then there exists a unique geodesic path $\sigma$ in $\Sigma_\Gamma$ connecting $p$ and $q$. By perturbing $p$ we may assume $\sigma$ does not pass through any vertices of $\Sigma_\Gamma$. Let $D(p,q)$ denote the union of all squares $S$ in $\Sigma_\Gamma$ such that $\sigma$ passes through the interior of $S$. Observe that $D(p,q)$ is topologically a disk. We  construct larger and larger disks containing $D(p, q)$. This is done in Lemmas 4.1-4.4.   The union of these disks will be a copy of $\mathbb H^2$ (an apartment) containing $P$ and $Q$,  verifying condition (1).  
	
	For the next few statements, we focus attention on the boundary of $D(p,q)$, which we will denote by $\partial D(p,q)$. For a finite subcomplex $D$ that is topologically a disk, a vertex $v$ on the boundary $\partial D$ is {\em convex, flat,} or {\em concave} if the angle at $v$ in $D$ is less than $\pi$, equal to $\pi,$ or greater than $\pi$ respectively. The following lemma says that the local structures of $D(p,q)$ at concave vertices are all the same.
	
	\begin{lemma}\label{m+1}
		Let $v$ be a concave vertex in $\partial D(p,q)$. Then $D(p,q)$ contains exactly $m+1$ squares incident to $v$.
	\end{lemma}
	\begin{proof}
		Let $v$ be concave in $\partial D(p,q)$ and suppose $D(p,q)$ contains $k$ consecutive squares incident to $v$ with $k\geq m+2$. Denote them by $S_1,\dots, S_k$ going counterclockwise. Let $v_1$ be the intersection of $\sigma$ with the common edge of $S_1,S_2$ and let $v_2$ be the intersection of $\sigma$ with the common edge of $S_{m+1},S_{m+2}$. The two points $v_1,v_2$ are not vertices of $\Sigma_\Gamma$ given that $\sigma$ does not pass through any vertices of $\Sigma_\Gamma$. In particular, they must be distinct from $v$. Then the union $v_1v\cup vv_2$ is a geodesic in $\Sigma_\Gamma$ (see Figure \ref{concave fig} for the case when $m=3$). Since $v_1,v_2\in\sigma$, the unique geodesic property of CAT($-1$) spaces implies that $v_1v\cup vv_2\subset\sigma$, contradicting the fact that $\sigma$ does not pass through any vertex.
	\end{proof}
	
	\begin{lemma}\label{flats}
		Between any two concave vertices on $\partial D(p,q)$ there is at least one convex vertex.
	\end{lemma}
	\begin{proof}
		Let $v,v'$ be two concave vertices in $\partial D(p,q)$ such that one of the two components of $\partial D(p, q)\backslash\{v, v'\}$ contains only convex and flat vertices. First suppose $v$ and $v'$ are adjacent. By Lemma \ref{m+1} we can find exactly $m+1$ squares in $D(p,q)$ incident to $v$ and $m+1$ squares in $D(p,q)$ incident to $v'$. Denote them by $S_1,\dots, S_{m+1}$ and $T_1,\dots, T_{m+1}$ respectively going counterclockwise such that $S_{m+1}=T_1$ (see Figure \ref{convbtw} for the case where $m=3$). We make use of the argument from the proof of Lemma \ref{m+1}. Let $e$ be the common edge of $S_1,S_2$ and $e'$ the common edge of $T_m,T_{m+1}$. Let $v_1,v_2$ be the intersection of $\sigma$ with the edges $e,e'$ respectively. As before, $v_1,v_2$ are not vertices in $\Sigma_\Gamma$ and the path $v_1v\cup vv'\cup v'v_2$  is a geodesic in $\Sigma_\Gamma$. By the unique geodesic property of CAT($-1$) spaces it follows that $v_1v\cup vv'\cup v'v_2\subset\sigma$. This contradicts the assumption that $\sigma$ does not pass through any vertices of $\Sigma_\Gamma$. Therefore, there must be vertices $v_1,\dots,v_k$ in $\partial D(p,q)$ that lie between $v$ and $v'$, each of which must be either convex or flat. If we suppose that each of the $v_i$ is flat (again see Figure \ref{convbtw} for the case where $m=3$), then a similar argument shows that $\sigma$ must pass through a vertex of $\Sigma_\Gamma$. Hence, there must be some $i$ such that $v_i$ is convex. 
	\end{proof}

	\begin{figure}
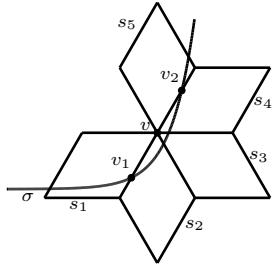


		\caption{The case $m=3$.}
		\label{concave fig}
	\end{figure}
	
	\begin{figure}
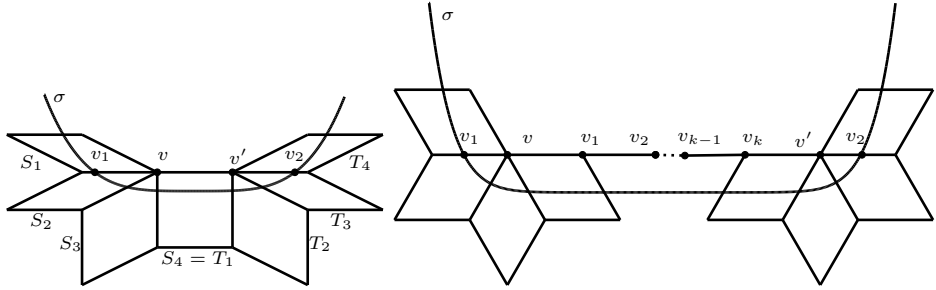

		%

		\caption{The two images illustrate the cases described in the proof of Lemma \ref{flats}. The image on the left shows why two adjacent vertices cannot be concave. The image on the right shows the case where all of the vertices between $v$ and $v'$ are flat. }\label{convbtw}
	\end{figure}

A finite subcomplex $D$ of $\Sigma_\Gamma$ that is also a topological disk is called a {\it{convex disk}}  if\begin{enumerate}
\item every vertex on $\partial D$ is either flat or convex;  and
\item every vertex $v$ in the interior of $D$ is flat, which means there are exactly $2m$ squares in $D$  that are incident to $v$.
\end{enumerate}

Note that for any subcomplex $D$ of $\Sigma_\Gamma$ that is also a disk, there are at least 3 convex vertices on the boundary. Indeed, letting $\alpha_1,\dots,\alpha_m$ be the angles in $D$ at convex vertices (so $0<\alpha_i<\pi$) and $\beta_1,\dots,\beta_n$ the angles in $D$ at concave vertices (so $\beta_j>\pi$), it follows from Gauss-Bonnett that \[\sum_i(\pi-\alpha_i)+\sum_j(\pi-\beta_j)-A\ge 2\pi,\] where $A$ is the area of the disk $D$. This implies that $\sum_i(\pi-\alpha_i)=2\pi+A+\sum_j(\beta_j-\pi)>2\pi$ and so $m\geq 3$.

This fact is necessary in establishing the next two lemmas. In particular, we have that $\partial D(p,q)$ must contain at least three convex vertices. We next show that the disk $D(p, q)$ is contained in a convex disk. 

\begin{lemma}\label{disk extension}
There exists a convex disk $D'$ containing $D(p,q)$.
\end{lemma}
\begin{proof}
Notice that $D(p, q)$ contains no interior vertices. If $\partial D(p,q)$ contains no concave vertices, then we are done. So suppose there is a concave vertex $v\in\partial D(p,q)$. Denote the remaining vertices on $\partial D(p,q)$ by $v_1, \dots, v_k$ proceeding counterclockwise from $v$. As $v$ is a concave vertex, there are exactly $m+1$ squares in $D(p,q)$ incident to $v$. They give rise to an edge path of combinatorial length $m+1$ in the link of $v$: $e_1,e_2,\dots, e_{m+1}$. Since the vertex link at each vertex in $\Sigma_\Gamma$ is isomorphic to $\Gamma$, there is a cycle $C$ of combinatorial length $2m$ in the link of $v$ that contains $e_1$ and $e_{m+1}$ by definition of a generalized $m$-gon. $C$ then contains $e_2,\dots, e_{m}$ since every cycle of combinatorial length $2m$ in $\Gamma$ is convex in $\Gamma$. Let $e_{m+2},\dots,e_{2m}$ denote the other $m-1$ edges in $C$. They correspond to $m-1$ squares $S_{m+2},\dots,S_{2m}$ in $\Sigma_\Gamma$ incident to $v$. Let $D_1$ be the union of $D(p,q)$ and $S_{m+2},\dots, S_{2m}$. Then $D_1$ is a disk and the vertex $v$ is a flat vertex contained in the interior of $D_1$.

Note that while we have already shown that $v_1,v_k$ cannot be concave in $\partial D(p,q)$, the construction of $D_1$ will result in the angles at $v_1, v_k$ increasing. If both $v_1,v_k$ are convex vertices in $\partial D(p,q)$, then they are either convex or flat in $\partial D_1$ and thus $\partial D_1$ has one less concave vertex than $\partial D(p,q)$. So suppose $v_1$ is flat in $\partial D(p,q)$. Then $v_1$ is concave in $\partial D_1$. Again, the vertex link of each vertex in $\Sigma_\Gamma$ is isomorphic to $\Gamma$, so we may repeat the process we used with $v$ at $v_1$ to extend $D_1$. Doing so yields a disk $D_2$ which contains $D_1$ and $v_1$ is a flat vertex contained in the interior of $D_2$. We then check if $v_2$ is concave in $\partial D_2$, continuing counterclockwise in this manner. The previous lemma and our above note ensure that the process will terminate at some $v_i$ which is convex in $\partial D(p,q)$. Repeating this procedure with $v_k$ traveling clockwise as needed, the result is a disk $D'$ containing $D(p,q)$ with less concave vertices in its boundary than in $\partial D(p,q)$.

We repeat the procedure with each concave vertex of $\partial D'$, the process eventually terminating with a convex disk $D''$ containing $D(p,q)$.
\end{proof}

\begin{lemma}
For every convex disk $D\subset \Sigma_\Gamma$ and every vertex $v$ in $\partial D$ there exists a convex disk $D'\subset\Sigma_\Gamma$ containing $D$ such that $v$ lies in the interior of $D'$.
\end{lemma}
\begin{proof}
Let $D$ be a convex disk in $\Sigma_\Gamma$ and take $v\in\partial D$. As previously noted there are at least two convex vertices on $\partial D$ that are different from $v$ which itself is either convex or flat in $\partial D$. Thus, there are k squares in $D$ incident to $v$ with $1\leq k<m+1$, giving rise to an edge path of combinatorial length $k$ in the link of $v$: $e_1,e_2,\dots, e_k$. Again, the vertex link at $v$ is isomorphic to $\Gamma$, implying the existence of a cycle $C$ of length $2m$ in the link of $v$ that contains $e_1,\dots,e_k$. Let $e_{k+1},\dots,e_{2m}$ be the remaining $2m-k$ edges in $C$. These edges correspond to $2m-k$ squares $S_{k+1},\dots,S_{2m}$ in $\Sigma_\Gamma$ incident to $v$. We let $D_1$ be the union of $D$ and the squares $S_{k+1},\dots,S_{2m}$. Then $D_1$ is topologically a disk. We now extend $D_1$ as in the proof of Lemma \ref{disk extension} on both sides of $v$ until we reach convex vertices (there are at least three convex vertices as we observed above). The resulting disk is the desired convex disk.
\end{proof}

\begin{theorem}\label{Thm}
Let $\Gamma$ be the graph of a finite generalized thick $m$-gon with $m\geq 3$. Then the Davis complex $\Sigma_\Gamma$ associated with the right-angled Coxeter group $W_\Gamma$ admits a Fuchsian building structure.
\end{theorem}
\begin{proof}
Using our previous notation we take $R$ to be a regular 4-gon in $\mathbb{H}^2$ with angle $\pi/m$ at each vertex and edges labeled \{1\},\{2\},\{3\},\{4\} proceeding counterclockwise. Thus $A_R$ will be a tessellation of $\mathbb{H}^2$ by the images of $R$ under $W$, the Coxeter group generated by the reflections about the edges of $R$. Considering $\Sigma_\Gamma$ as a labeled 2-complex (see Section 3), we need only verify that $\Sigma_\Gamma$ satisfies the properties listed in Definition \ref{fuch}. 

Property 2 is clearly satisfied and Lemma \ref{degree} gives us the existence of the necessary integers $q_i,i=1,2,3,4$. All that remains is to show that given two squares $P,Q\subset\Sigma_\Gamma$ we can find a copy of $A_R$ containing them. Choose two interior points $p\in P,q\in Q$ and join them by a unique geodesic $\sigma$ in $\Sigma_\Gamma$. By perturbing $p$ we may assume $\sigma$ does not pass through any vertices of $\Sigma_\Gamma$ as before. Taking $D(p,q)$ to be as previously defined, we obtain a nested sequence of convex disks whose union is a copy of $A_R$ that contains both $P$ and $Q$. Thus, $\Sigma_\Gamma$ admits a Fuchsian building structure. 
\end{proof}
We can now prove our main result.
\begin{proof}[Proof of Theorem \ref{main}]
As $\Sigma_{\Gamma_i}$, $i=1,2$, is a Fuchsian building by Theorem \ref{Thm}, it follows from Theorem \ref{xie} that any quasi-isometry $f:\Sigma_{\Gamma_1}\rightarrow \Sigma_{\Gamma_2}$ is at a finite distance from an isomorphism. In particular, if $W_{\Gamma_1}$ and $W_{\Gamma_2}$ are quasi-isometric, then there exists a quasi-isometry $f:\Sigma_{\Gamma_1}\rightarrow \Sigma_{\Gamma_2}$. As previously stated, $f$ lies at finite distance from an isomorphism implying that $\Sigma_{\Gamma_1}$ and $\Sigma_{\Gamma_2}$ are isomorphic. Recall that the vertex links of $\Sigma_{\Gamma_i}$ are $\Gamma_i$. Since $\Sigma_{\Gamma_1}$ and $\Sigma_{\Gamma_2}$ are isomorphic, the vertex links $\Gamma_1$ and $\Gamma_2$ must also be isomorphic. Conversely, if $\Gamma_1,\Gamma_2$ are isomorphic, then clearly $W_{\Gamma_1},W_{\Gamma_2}$ are isomorphic, and so are quasi-isometric.
\end{proof}

The following theorem due to Kapovitch-Kleiner-Leeb allows us to prove Corollary \ref{cor2}.

\begin{theorem}[Theorem B in \cite{KKL98}]\label{prod}
Suppose $M=Z\times\prod\limits_{i=1}^{k}M_i$ and $N=W\times\prod\limits_{i=1}^{l}N_i$ are geodesic metric spaces such that the asymptotic cones of $Z$ and $W$ are homeomorphic to $\mathbb{R}^n$ and $\mathbb{R}^m$ respectively, and the components $M_i,N_j$ are of coarse type I or II. Then for every $L\geq 1,A\geq 0$ there is a constant $D$ so that for each $(L,A)$-quas-isometry $\phi:M\rightarrow N$ we have $k=l,n=m$ and after reindexing the factors $N_j$ there are quasi-isometries $\phi_i:M_i\rightarrow N_i$ such that for every $i$ the following diagram commutes up to error at most $D$:
\begin{equation*}
\begin{tikzcd}
M \arrow{r}{\phi} \arrow[swap]{d}{} & N \arrow{d}{} \\
M_i \arrow{r}{\phi_i} & N_i
\end{tikzcd}
\end{equation*}
\end{theorem}
We refer the reader to \cite{KL96} for precise definitions of asymptotic cones and coarse types. We note that the asymptotic cone of a Gromov hyperbolic space is an  $\mathbb R$-tree. Thus given a generalized $m$-gon with $m\geq 3$ the corresponding Davis complex $\Sigma_\Gamma$ is of coarse type I. 
\begin{proof}[Proof of Corollary \ref{cor2}]
Say $\Gamma_1=\mathcal M_1\star \mathcal M_2\star\cdots\star\mathcal M_k$ and  $\Gamma_2=\mathcal N_1\star\mathcal N_2\star\cdots\star\mathcal N_l$ with $\mathcal M_i,\mathcal N_j\in\mathcal S$. As previously stated in Section 2.1, the corresponding right-angled Coxeter groups $W_{\Gamma_1},W_{\Gamma_2}$ will then be the direct products $W_{\Gamma_1}=\prod_{i=1}^{k}W_{\mathcal M_i}$ and $W_{\Gamma_2}=\prod_{i=1}^{l}W_{\mathcal N_i}$. It then follows that $\Sigma_{\Gamma_1}=\prod_{i=1}^{k}\Sigma_{\mathcal M_i}$ and $\Sigma_{\Gamma_2}=\prod_{i=1}^{l}\Sigma_{\mathcal N_i}$.

First suppose $f:\Sigma_{\Gamma_1}\rightarrow \Sigma_{\Gamma_2}$ is a quasi-isometry. As all of the hypotheses of Theorem \ref{prod} are clearly satisfied with $Z$ and $W$ homeomorphic to $\mathbb{R}^0$, we have $k=l$ and, after reindexing, there are quasi-isometries $\phi_i:W_{\mathcal M_i}\rightarrow W_{\mathcal N_i}$ such that $f$ is at finite distance from $(\phi_1,\phi_2,\dots,\phi_k)$. By Theorem \ref{main}, each $\phi_i$ is at a finite distance from an isomorphism. This implies that $f$ itself will be at a finite distance from an isomorphism.

In particular, if $W_{\Gamma_1}$ and $W_{\Gamma_2}$ are quasi-isometric, then $\Sigma_{\Gamma_1}$ and $\Sigma_{\Gamma_2}$ are quasi-isometric. Thus after reindexing we have that $k=l$ and $W_{\mathcal M_i}$ is quasi-isometric to $W_{\mathcal N_i}$, $i=1,\dots, k$. By Theorem \ref{main}, $\Sigma_{\mathcal M_i}$ is isomorphic to $\Sigma_{\mathcal N_i}$ and $M_i$ is isomorphic to $N_i$ for each $i$. It must then be the case that $\Gamma_1$ and $\Gamma_2$ are isomorphic. Conversely, if $\Gamma_1$ and $\Gamma_2$ are isomorphic, then clearly $W_{\Gamma_1}$ and $W_{\Gamma_2}$ are isomorphic, and so are quasi-isometric.
\end{proof}

\end{section}

\begin{section}{Constructing commensurable RACGs}
\allowdisplaybreaks
In this section we describe a method for constructing examples of commensurable RACGs. This method is inspired by a construction of Bestvina-Kleiner-Sageev in the case of right-angled Artin groups (see Example 1.4 of \cite{BKS}). Let $\Gamma$ be a finite simplicial graph and $K\subset\Gamma$ a proper clique. We construct a finite simplicial graph $\Gamma'$ from $\Gamma$ such that $W_{\Gamma'}$ is isomorphic to a finite index subgroup of $W_\Gamma$. In particular, it will follow immediately that $W_\Gamma$ and $W_{\Gamma'}$ are quasi-isometric.

Denote the vertex sets of $\Gamma$ and $K$ by $V(\Gamma)$ and $V(K)$ respectively. A {\em reduced word} for an element $g\in W_\Gamma$ is a word $w$ in $V(\Gamma)$ of minimal length that represents $g$. It is well known that $w=u_1u_2\cdots u_n\in W_\Gamma$ is reduced if and only if it can not be shortened via a sequence of operations of either deleting a consecutive pair of the form $uu$, $u\in V(\Gamma)$, or transposing two consecutive letters $u,v$ such that $uv=vu$ in $W_\Gamma$ (see Section 3.4 of \cite{Dav08}).

Define $\rho:V(\Gamma)\rightarrow V(K)\cup\{1\}$ by \[ \rho(v)=\begin{cases}
v\hspace{.5cm}\text{if}\ v\in V(K)\\
1\hspace{.5cm}\text{otherwise}.
\end{cases}\] Note that $\rho(v^2)=1$ for all $v\in V(\Gamma)$. Moreover, given $v_1,v_2\in V(\Gamma)$ such that $v_1v_2=v_2v_1$ we clearly have $\rho(v_1)\rho(v_2)=\rho(v_2)\rho(v_1)$. Hence $\rho$ is a map between generating sets that preserves relators and can therefore be extended to a homomorphism $\phi:W_\Gamma\rightarrow W_K$. Recall that $W_K$ is finite since $K$ is a clique, making $\ker\phi$ a finite index subgroup of $W_\Gamma$. We construct $\Gamma'$ so that $W_{\Gamma'}$ is isomorphic to $\ker\phi$. As a result, $W_\Gamma$ and $W_{\Gamma'}$ will be commensurable. We begin by focusing attention on finding a useful generating set for $\ker\phi$.

Enumerate the vertices of $K$ as $V(K)=\{v_1,\dots,v_k\}$. As $K$ is a clique we have $v_iv_j=v_jv_i$ for every $v_i,v_j\in V(K)$. Thus every element of $W_K$ can be written in the form $v_1^{\epsilon_1}v_2^{\epsilon_2}\cdots v_k^{\epsilon_k}$ where $\epsilon_i\in\mathbb Z_2$ for all $i$. Hence there is a clear bijection $g:\mathbb Z_2^k\rightarrow W_K$ defined by $g(\epsilon_1,\dots,\epsilon_k)=v_1^{\epsilon_1}\cdots v_k^{\epsilon_k}$. Consider the sets $T=\{g(\epsilon)vg(\epsilon):\ v\in V(\Gamma)\setminus V(K),\ \epsilon\in\mathbb Z_2^k\}$ and $R=\{g(\epsilon)vg(\epsilon)\in T:\ g(\epsilon)vg(\epsilon)\ \text{is reduced in}\ W_\Gamma\}$. Note that given $g(\epsilon)vg(\epsilon)\in T\setminus R$ there must be some $j$ with $\epsilon_j=1$ such that $(v,v_j)\in E(\Gamma)$. Then $v,v_j$ commute in $W_\Gamma$ and we have \begin{align*} v_1^{\epsilon_1}\cdots v_j\cdots v_k^{\epsilon_k}vv_1^{\epsilon_1}\cdots v_j\cdots v_k^{\epsilon_k}&=v_1^{\epsilon_1}\cdots v_{j-1}^{\epsilon_{j-1}}v_{j+1}^{\epsilon_{j+1}}\cdots v_k^{\epsilon_k}v_jvv_jv_1^{\epsilon_1}\cdots v_{j-1}^{\epsilon_{j-1}}v_{j+1}^{\epsilon_{j+1}}\cdots v_k^{\epsilon_k}\\
&=v_1^{\epsilon_1}\cdots v_{j-1}^{\epsilon_{j-1}}v_{j+1}^{\epsilon_{j+1}}\cdots v_k^{\epsilon_k}vv_1^{\epsilon_1}\cdots v_{j-1}^{\epsilon_{j-1}}v_{j+1}^{\epsilon_{j+1}}\cdots v_k^{\epsilon_k}\; \text{in}\; W_\Gamma.\end{align*} It follows inductively that there exists $\delta=(\delta_1,\dots,\delta_k)\in\mathbb Z_2^k$ such that $g(\epsilon)vg(\epsilon)=g(\delta)vg(\delta)$ in $W_\Gamma$ and $g(\delta)vg(\delta)$ is reduced. Therefore $R$ generates $\langle T\rangle$.

\begin{lemma}
$\ker\phi=\langle R\rangle$.
\end{lemma}
\begin{proof}
Given $g(\epsilon_1)u_1g(\epsilon_1)g(\epsilon_2)u_2g(\epsilon_2)\cdots g(\epsilon_p)u_pg(\epsilon_p)\in\langle R\rangle$ we have \begin{align*} &\phi(g(\epsilon_1)u_1g(\epsilon_1)g(\epsilon_2)u_2g(\epsilon_2)\cdots g(\epsilon_p)u_pg(\epsilon_p))\\
&=g(\epsilon_1)\phi(u_1)g(\epsilon_1)g(\epsilon_2)\phi(u_2)g(\epsilon_2)\cdots g(\epsilon_p)\phi(u_p)g(\epsilon_p)\\
&=g(\epsilon_1)g(\epsilon_1)g(\epsilon_2)g(\epsilon_2)\cdots g(\epsilon_p)g(\epsilon_p)=1.\end{align*} Thus $\langle R\rangle\subset\ker\phi$. To show the reverse containment let $w\in\ker\phi\setminus\{1\}$. Then there exist $u_1,u_2,\dots,u_n\in V(\Gamma), n>0,$ such that $u_i\neq u_{i+1}$ for all $i$ and $w=u_1u_2\cdots u_n$. We proceed by induction on $n$.

When $n=1$ we have $1=\phi(w)=\phi(u_1)$. Hence $u_1\in V(\Gamma)\setminus V(K)$ and $w=g(0)u_1g(0)\in\langle T\rangle$ where $0=(0,\dots,0)\in\mathbb Z_2^k$. So suppose $n>1$ and that every element of $\ker\phi$ that is a product of at most $n-1$ elements of $V(\Gamma)$ lies in $\langle T\rangle$. If $u_1\in V(\Gamma)\setminus V(K)$, then $1=\phi(u_1)\phi(w)=\phi(u_1w)=\phi(u_2\cdots u_n)$. Hence $u_2\cdots u_n\in\ker\phi$, has length $n-1$, and therefore $u_2\cdots u_n\in\langle T\rangle$ by the inductive hypothesis. Since $u_1\in V(\Gamma)\setminus V(K)$ it follows that $u_1=g(0)u_1g(0)\in\langle T\rangle$ and $w\in\langle T\rangle$. 

Now suppose $u_1\in V(K)$. Let $i\ge 1$ be the smallest integer such that $u_{i+1}\notin V(K)$ and $j>i+1$ the smallest integer such that $u_j\in V(K)$. For each $t$ with $1\leq t\leq j-1-i$ set $w_t=(u_1\cdots u_i)u_{i+t}(u_1\cdots u_i)$. Note that since $u_1,\dots,u_i\in V(K)$ we have $(u_1\cdots u_i)^2=u_1\cdots u_iu_1\cdots u_i=u_1^2\cdots u_i^2=1$. Then
\begin{align*}
w&=u_1\cdots u_iu_{i+1}\cdots u_j\cdots u_n\\
&=(u_1\cdots u_i)u_{i+1}(u_1\cdots u_i)(u_1\cdots u_i)u_{i+2}\cdots u_j\cdots u_n\\
&=w_1(u_1\cdots u_i)u_{i+2}\cdots u_j\cdots u_n\\
&=w_1(u_1\cdots u_i)u_{i+2}(u_1\cdots u_i)(u_1\cdots u_i)u_{i+3}\cdots u_j\cdots u_n\\
&=w_1w_2(u_1\cdots u_i)u_{i+3}\cdots u_j\cdots u_n\\
&\hspace{2cm}\vdots\\
&=w_1w_2\cdots w_{j-1-i}u_1\cdots u_iu_j\cdots u_n.\end{align*} 

Each $u_{i+t}\in V(\Gamma)\setminus V(K)$ for $1\leq t\leq j-1-i$ implying $\phi(w_t)=(u_1\cdots u_i)\cdot1\cdot(u_1\cdots u_i)=(u_1\cdots u_i)^2=1$ for all $t$. Therefore, \[ 1=\phi(w)=\phi(w_1w_2\cdots w_{j-1-i}u_1\cdots u_iu_j\cdots u_k)=\phi(u_1\cdots u_iu_j\cdots u_k).\] The element $u_1\cdots u_iu_j\cdots u_k$ then lies in $\ker\phi$ and has length at most $n-1$. It follows from the inductive hypothesis that $u_1\cdots u_iu_j\cdots u_k\in\langle T\rangle$. Thus $w\in\langle T\rangle$. As previously noted, $R$ generates $\langle T\rangle$. Hence $w\in\langle R\rangle$ and $\ker\phi\subset\langle R\rangle$.
\end{proof}

We now construct the graph $\Gamma'$. Let $R'$ be a set and $h:R\rightarrow R'$ a bijection. Denote $h(g(\epsilon)vg(\epsilon))$ by $v(\epsilon)$. Take $R'$ to be the vertex set of $\Gamma'$. Two vertices $u(\epsilon_1)$ and $v(\epsilon_2)$ are joined by an edge in $\Gamma'$ if $(u,v)\in E(\Gamma)$ and there is some $\epsilon\in\mathbb Z_2^k$ such that $g(\epsilon_1)u g(\epsilon_1)=g(\epsilon)u g(\epsilon)$ in $W_\Gamma$
and  $g(\epsilon_2)v g(\epsilon_2)=g(\epsilon)v g(\epsilon)$ in $W_\Gamma$. From this construction $\Gamma'$ is finite and simplicial.

Visually, $\Gamma'$ is formed by joining together multiple isomorphic copies of the graph $\Lambda$, the graph formed by removing the subgraph $K$ along with all edges incident to a vertex in $K$ from $\Gamma$. For each $\epsilon\in\mathbb Z_2^k$ define the graph $\Lambda_\epsilon$ to have vertex set $T_\epsilon=\{g(\epsilon)vg(\epsilon):v\in V(\Gamma)\setminus V(K)\}$ and an edge joining vertices $g(\epsilon)vg(\epsilon),g(\epsilon)ug(\epsilon)$ if $v,u$ are adjacent in $\Lambda$. $\Gamma'$ is obtained by taking the union $\cup_{\epsilon\in\mathbb Z_2^k}\Lambda_{\epsilon}$ and identifying the vertices $g(\epsilon)vg(\epsilon)$ and $g(\delta)ug(\delta)$ if $g(\epsilon)vg(\epsilon)=g(\delta)ug(\delta)$ in $W_\Gamma$. See Figure \ref{gamma'} for an example.

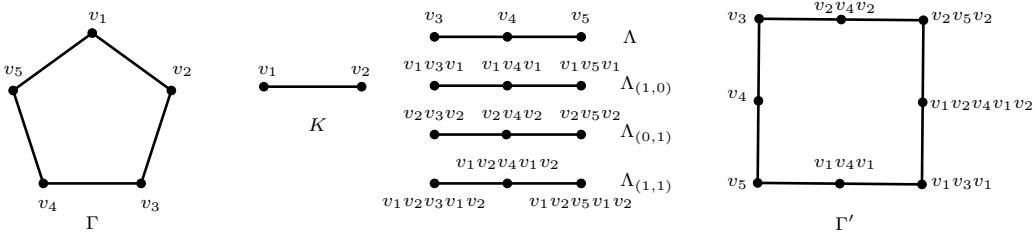
\begin{figure}
\begin{tikzpicture}[scale=0.65]
\fill[line width=1pt,color=black,fill=white] (-5,0) -- (-3,0) -- (-2.381966011250105,1.9021130325903064) -- (-4,3.077683537175253) -- (-5.618033988749895,1.9021130325903073) -- cycle;
\fill[line width=1pt,color=black,fill=white] (9.61,0.01) -- (12.97,-0.02) -- (13,3.34) -- (9.64,3.37) -- cycle;
\draw [line width=1pt,color=black] (-5,0)-- (-3,0);
\draw [line width=1pt,color=black] (-3,0)-- (-2.381966011250105,1.9021130325903064);
\draw [line width=1pt,color=black] (-2.381966011250105,1.9021130325903064)-- (-4,3.077683537175253);
\draw [line width=1pt,color=black] (-4,3.077683537175253)-- (-5.618033988749895,1.9021130325903073);
\draw [line width=1pt,color=black] (-5.618033988749895,1.9021130325903073)-- (-5,0);
\draw [line width=1pt] (-0.49,1.98)-- (1.51,1.98);
\draw [line width=1pt] (3,0)-- (6,0);
\draw [line width=1pt] (3,1)-- (6,1);
\draw [line width=1pt] (3,2)-- (6,2);
\draw [line width=1pt] (3,3)-- (6,3);
\draw [line width=1pt,color=black] (9.61,0.01)-- (12.97,-0.02);
\draw [line width=1pt,color=black] (12.97,-0.02)-- (13,3.34);
\draw [line width=1pt,color=black] (13,3.34)-- (9.64,3.37);
\draw [line width=1pt,color=black] (9.64,3.37)-- (9.61,0.01);
\begin{scriptsize}
\draw [fill=black] (-5,0) circle (2.5pt);
\draw[color=black] (-4.9,-0.45) node {$v_4$};
\draw [fill=black] (-3,0) circle (2.5pt);
\draw[color=black] (-2.8,-0.45) node {$v_3$};
\draw[color=black] (-4,-0.8) node {$\Gamma$};
\draw [fill=black] (-2.381966011250105,1.9021130325903064) circle (2.5pt);
\draw[color=black] (-2.15,2.3) node {$v_2$};
\draw [fill=black] (-4,3.077683537175253) circle (2.5pt);
\draw[color=black] (-3.87,3.4) node {$v_1$};
\draw [fill=black] (-5.618033988749895,1.9021130325903073) circle (2.5pt);
\draw[color=black] (-5.6,2.3) node {$v_5$};
\draw [fill=black] (-0.49,1.98) circle (2.5pt);
\draw[color=black] (-0.4,2.3) node {$v_1$};
\draw [fill=black] (1.51,1.98) circle (2.5pt);
\draw[color=black] (1.5,2.3) node {$v_2$};
\draw[color=black] (0.61,1.2) node {$K$};
\draw [fill=black] (3,0) circle (2.5pt);
\draw[color=black] (3,-0.4) node {$v_1v_2v_3v_1v_2$};
\draw [fill=black] (6,0) circle (2.5pt);
\draw[color=black] (6,-0.4) node {$v_1v_2v_5v_1v_2$};
\draw[color=black] (7.35,0) node {$\Lambda_{(1,1)}$};
\draw [fill=black] (3,1) circle (2.5pt);
\draw[color=black] (3,1.4) node {$v_2v_3v_2$};
\draw [fill=black] (6,1) circle (2.5pt);
\draw[color=black] (6.2,1.4) node {$v_2v_5v_2$};
\draw[color=black] (7.35,1) node {$\Lambda_{(0,1)}$};
\draw [fill=black] (3,2) circle (2.5pt);
\draw[color=black] (3,2.4) node {$v_1v_3v_1$};
\draw [fill=black] (6,2) circle (2.5pt);
\draw[color=black] (6.2,2.4) node {$v_1v_5v_1$};
\draw[color=black] (7.35,2) node {$\Lambda_{(1,0)}$};
\draw [fill=black] (3,3) circle (2.5pt);
\draw[color=black] (3,3.4) node {$v_3$};
\draw [fill=black] (6,3) circle (2.5pt);
\draw[color=black] (6,3.4) node {$v_5$};
\draw[color=black] (7,3) node {$\Lambda$};
\draw [fill=black] (4.49,3) circle (2.5pt);
\draw[color=black] (4.5,3.4) node {$v_4$};
\draw [fill=black] (4.49,2) circle (2.5pt);
\draw[color=black] (4.6,2.4) node {$v_1v_4v_1$};
\draw [fill=black] (4.47,1) circle (2.5pt);
\draw[color=black] (4.6,1.4) node {$v_2v_4v_2$};
\draw [fill=black] (4.49,0) circle (2.5pt);
\draw[color=black] (4.5,0.4) node {$v_1v_2v_4v_1v_2$};
\draw [fill=black] (9.61,0.01) circle (2.5pt);
\draw[color=black] (9.2,0) node {$v_5$};
\draw [fill=black] (12.97,-0.02) circle (2.5pt);
\draw[color=black] (13.8,0) node {$v_1v_3v_1$};
\draw [fill=black] (13,3.34) circle (2.5pt);
\draw[color=black] (13.8,3.4) node {$v_2v_5v_2$};
\draw [fill=black] (9.64,3.37) circle (2.5pt);
\draw[color=black] (9.2,3.4) node {$v_3$};
\draw [fill=black] (9.625,1.69) circle (2.5pt);
\draw[color=black] (9.2,1.8) node {$v_4$};
\draw [fill=black] (11.32,3.355) circle (2.5pt);
\draw[color=black] (11.4,3.6) node {$v_2v_4v_2$};
\draw [fill=black] (12.985,1.66) circle (2.5pt);
\draw[color=black] (14.2,1.6) node {$v_1v_2v_4v_1v_2$};
\draw [fill=black] (11.29,-0.005) circle (2.5pt);
\draw[color=black] (11.4,0.4) node {$v_1v_4v_1$};
\draw[color=black] (11.4,-0.8) node {$\Gamma'$};
\end{scriptsize}
\end{tikzpicture}
\caption{Let $\Gamma$ be the given five cycle and $K$ the clique consisting of the vertices $v_1,v_2$ and their shared edge. We obtain $\Gamma'$ by taking the union of the graphs $\Lambda, \Lambda_{(1,0)},\Lambda_{(0,1)},\Lambda_{(1,1)}$ and identifying the vertices $v_3=v_2v_3v_2$, $v_2v_5v_2=v_1v_2v_5v_1v_2v_5$, $v_1v_3v_1=v_1v_2v_3v_1v_2$, and $v_5=v_1v_5v_1$. The result is the 8-cycle pictured above.}\label{gamma'}
\end{figure}

\begin{prop}\label{ker} $W_\Gamma$ and $W_{\Gamma'}$ are commensurable.
\end{prop}
\begin{proof}
As stated previously in this section it suffices to show that $\ker\phi\cong W_\Gamma'$. By construction there is a bijection $h^{-1}:R'\rightarrow R$ given by $h^{-1}(v(\epsilon))=g(\epsilon)vg(\epsilon)$. Certainly $h^{-1}(v(\epsilon)v(\epsilon))=(g(\epsilon)vg(\epsilon))^2=g(\epsilon)vg(\epsilon)g(\epsilon)vg(\epsilon)=1$ for all $v(\epsilon)\in R'$. Additionally, given $v(\epsilon_1),u(\epsilon_2)\in R'$ with $(v(\epsilon_1),u(\epsilon_2))\in E(\Gamma')$ we have $(u,v)\in E(\Gamma)$ and there must be some $\epsilon\in\mathbb Z_2^k$ such that $g(\epsilon_1)vg(\epsilon_1)=g(\epsilon)vg(\epsilon)$ in $W_\Gamma$ and $g(\epsilon_2)ug(\epsilon_2)=g(\epsilon)ug(\epsilon)$ in $W_\Gamma$. Thus \begin{align*} &h^{-1
}(v(\epsilon_1)u(\epsilon_2)v(\epsilon_1)u(\epsilon_2))\\
&=g(\epsilon_1)vg(\epsilon_1)g(\epsilon_2)ug(\epsilon_2)g(\epsilon_1)vg(\epsilon_1)g(\epsilon_2)ug(\epsilon_2)\\
&=g(\epsilon)vg(\epsilon)g(\epsilon)ug(\epsilon)g(\epsilon)vg(\epsilon)g(\epsilon)ug(\epsilon)\\
&=g(\epsilon)vuvug(\epsilon)=1\ \text{in}\ W_\Gamma.
\end{align*}
As $h^{-1}$ preserves the relators of $W_{\Gamma'}$ it can be extended to a homomorphism $H:W_{\Gamma'}\rightarrow\ker\phi$. It follows from Lemma 5.1 that $H$ is a surjection. All that remains is to show that $H$ is injective.

Suppose by way of contradiction that $\ker H$ is nontrivial. Take $w=u_1(\epsilon_1)\cdots u_p(\epsilon_p)$ $\in\ker H\setminus\{1\}$ of shortest length. Note that $p\neq 1$ for if it did, then $1=H(w)=H(u_1(\epsilon_1))=g(\epsilon_1)u_1g(\epsilon_1)\in R$ which is not possible by definition of $R$. So $p>1$. Since $1=H(w)=H(u_1(\epsilon_1)\cdots u_p(\epsilon_p))=g(\epsilon_1)u_1g(\epsilon_1)\cdots g(\epsilon_p)u_pg(\epsilon_p)$ in $W_\Gamma$, we can reduce this expression for $H(w)$ through a sequence of deletion and transposition operations in order to obtain the empty word in $W_\Gamma$.

Rewriting the above expression we have \[H(w)=g(\epsilon_1)u_1g(\epsilon_1)\cdots g(\epsilon_p)u_pg(\epsilon_p)=g(\epsilon_1)u_1g(\epsilon_1+\epsilon_2)u_2g(\epsilon_2+\epsilon_3)\cdots g(\epsilon_{p-1}+\epsilon_p)u_pg(\epsilon_p).\] Set $\epsilon_0=\epsilon_{p+1}=0$ and $\delta_i=\epsilon_{i-1}+\epsilon_i$ for $1\le i\le p+1$. Let $v_1^{a_{i,1}}v_2^{a_{i,2}}\cdots v_k^{a_{i,k}}=g(\delta_i)$ for $1\le i\le p+1$. In order for $H(w)$ to equal $1$ there must be $s,t$, $t>s$, such that $u_s=u_t$, $u_su_j=u_ju_s$ for all $s<j<t$, and $u_sv_r^{i,r}=v_r^{i,r}u_s$ for all $s+1\le i\le t$, $1\le r\le k$. Then
\begin{align*}
&g(\epsilon_s)u_sg(\epsilon_s)g(\epsilon_{s+1})u_{s+1}g(\epsilon_{s+1})\cdots g(\epsilon_t)u_tg(\epsilon_t)\\
&=g(\epsilon_s)u_sg(\delta_{s+1})u_{s+1}g(\delta_{s+2})\cdots g(\delta_t)u_tg(\epsilon_t),\hspace{2.1cm}\text{with}\ \delta_i=\epsilon_{i-1}+\epsilon_i,\\
&=g(\epsilon_s)g(\delta_{s+1})u_{s+1}g(\delta_{s+2})\cdots g(\delta_t)u_su_tg(\epsilon_t),\hspace{2.1cm}\text{since}\ u_su_j=u_ju_s, u_sv_r^{i,r}=v_r^{i,r}u_s,\\
&=g(\epsilon_s)g(\delta_{s+1})u_{s+1}g(\delta_{s+2})\cdots g(\delta_t)g(\epsilon_t),\hspace{2.8cm}\text{because}\ u_s=u_t,\\
&=g(\epsilon_s)g(\epsilon_s)g(\epsilon_{s+1})u_{s+1}g(\epsilon_{s+1})\cdots g(\epsilon_{t-1})g(\epsilon_t)g(\epsilon_t),\hspace{1cm}\text{with}\ \delta_i=\epsilon_{i-1}+\epsilon_i,\\
&=g(\epsilon_{s+1})u_{s+1}g(\epsilon_{s+1})\cdots g(\epsilon_{t-1})u_{t-1}g(\epsilon_{t-1}).
\end{align*}

Let $\zeta$ be the element of $W_\Gamma$ given by $\zeta=g(\epsilon_{s+1})u_{s+1}g(\epsilon_{s+1})\cdots g(\epsilon_{t-1})u_{t-1}g(\epsilon_{t-1})$. Then
\begin{align*}
H(w)&=g(\epsilon_1)u_1g(\epsilon_1)\cdots g(\epsilon_s)u_sg(\epsilon_s) \cdots g(\epsilon_t)u_tg(\epsilon_t)\cdots g(\epsilon_p)u_pg(\epsilon_p)\\
&=g(\epsilon_1)u_1g(\epsilon_1)\cdots g(\epsilon_{s-1})u_{s-1}g(\epsilon_{s-1})\cdot\zeta\cdot g(\epsilon_{t+1})u_{t+1}g(\epsilon_{t+1})\cdots g(\epsilon_p)u_pg(\epsilon_p)\ \text{in}\ W_\Gamma.
\end{align*}
Let $\hat w$ be the element of $W_{\Gamma'}$ given by \[\hat w=u_1(\epsilon_1)\cdots u_{s-1}(\epsilon_{s-1}) u_{s+1}(\epsilon_{s+1})\cdots u_{t-1}(\epsilon_{t-1})u_{t+1}(\epsilon_{t+1})\cdots u_p(\epsilon_p).\] Then the above calculation shows that $H(\hat w)=H(w)=1$. Hence $\hat w\in\ker H$ and has length $p-2$ in $W_{\Gamma'}$. As $w$ was chosen to be an element of $\ker H\backslash\{1\}$ with shortest length it follows that $\hat w$ must be the empty word in $W_{\Gamma'}$ and $p=2$. Then $w=u_1(\epsilon_1)u_1(\epsilon_2)$ with $u_1v_j^{a_{2,j}}=v_j^{a_{2,j}}u_1$, $1\le j\le k$, by our previous observations. 

Note that \[v_j^{a_{2,j}}=v_j^{a_{1,j}+ a_{3,j}}=\begin{cases}
1\hspace{1cm} \text{if}\ a_{1,j}= a_{3,j}\\
v_j\hspace{.85cm} \text{if}\ a_{1,j}\ne a_{3,j}
\end{cases}.\] Suppose there is $j$ such that $v_j^{a_{2,j}}=v_j$. Then exactly one of $\epsilon_1,\epsilon_2$ has $j^{th}$ coordinate $1$, say $\epsilon_1$. But this implies that $g(\epsilon_1)u_1g(\epsilon_1)$ can be reduced, which is not possible given that $g(\epsilon_1)u_1g(\epsilon_1)\in R$ and is reduced by definition of $R$. It must then be the case that $v_j^{a_{2,j}}=1$ for all $j$. In other words, $a_{1,j}=a_{3,j}$ for all $j$ and $\epsilon_1=\epsilon_2$. As a result $w=u_1(\epsilon_1)u_1(\epsilon_1)=1$ in $W_{\Gamma'}$, contradicting our assumption that $w$ was nontrivial. Therefore $\ker f$ is trivial and $\ker\phi\cong W_{\Gamma'}$ as desired.
\end{proof}
\end{section}
\vspace{-.35cm}
\bibliographystyle{amsplain}

\end{document}